\documentclass[11pt, twoside, a4paper]{article}

\usepackage{latexsym}
\usepackage{amscd}
\usepackage{theorem}
\usepackage{pifont}
\usepackage{amsfonts}
\usepackage{xspace}
\usepackage{amssymb}
\usepackage{fancyhdr}
\usepackage{mathrsfs}
\usepackage{amsmath}%
\usepackage{graphics}%
\usepackage{graphicx}%
\usepackage{euscript}%
\usepackage{psfrag}%
\usepackage{empheq}%
\usepackage{upgreek}
\usepackage{eufrak}
\usepackage{mathtools}

{\theorembodyfont{\slshape} \newtheorem{theorem}{Theorem}[section]}
{\theorembodyfont{\slshape} \newtheorem{lemma}[theorem]{Lemma}}
{\theorembodyfont{\slshape} \newtheorem{proposition}[theorem]{Proposition}}
{\theorembodyfont{\slshape} \newtheorem{corollary}[theorem]{Corollary}}
{\theorembodyfont{\slshape} \newtheorem{definition}[theorem]{Definition}}
{\theorembodyfont{\slshape} }

\newenvironment{proof}{\noindent\textbf{Proof:\ }}{$\hfill{\bullet}$}

\numberwithin{equation}{section}

\title{\textsc{Weak Formulation of the Laplacian on the Full Shift Space}}                

\author{Shrihari Sridharan \\ {\tt shrihari@iisertvm.ac.in}  \bigskip \\ Sharvari Neetin Tikekar \\ {\tt sharvai.tikekar14@iisertvm.ac.in}
 \bigskip \\ 
 Indian Institute of Science Education and Research \\ Thiruvananthapuram(IISER-TVM), India.}     

\date{}

\begin{document}

\maketitle
\thispagestyle{empty}

\begin{abstract}
\noindent
We consider a Laplacian on the one-sided full shift space over a finite symbol set, which is constructed as a renormalized limit of finite difference operators. We propose a weak definition of this Laplacian, analogous to the one in calculus, by choosing test functions as those which have finite energy and vanish on various boundary sets. In the abstract setting of the shift space, the boundary sets are chosen to be the sets on which the finite difference operators are defined. We then define the Neumann derivative of functions on these boundary sets and establish a relation between three important concepts in analysis so far, namely, the Laplacian, the bilinear energy form and the Neumann derivative of a function. As a result, we obtain the Gauss-Green's formula analogous to the one in classical case. We conclude this paper by providing a sufficient condition for the Neumann boundary value problem on the shift space.   
\end{abstract} 
\bigskip \bigskip

\begin{tabular}{l c l}
\textbf{Keywords} & : & One-sided full shift space, \\
& &Weak formulation of the Laplacian,  \\ & & Energy form, Neumann derivative. \\
\textbf{AMS Subject Classifications} & : & 28A15, 37B10, 31E05.  
\end{tabular}
\bigskip 
\bigskip

\maketitle

\bigskip

\section{Introduction} 
One of the main aspects of analysis is the theory of calculus. Until recently, calculus intrinsically necessitated the underlying space to be smooth. During $1970$'s, it was observed that nature is abound with non-smooth objects like fractals, the term which was coined by Mandelbrot \cite{mandelbrot_book,mandelbrot}. The need for studying various physical phenomena like heat and wave propagation on such sets gave birth to the theory of calculus on these rough surfaces, popularly known as \textit{rough analysis}. There are two main approaches towards constructing a Laplacian on fractals like the Sierpi\'nski gasket. The probabilistic approach \cite{goldstein, kusuoka, barperk} derives a Laplacian as a generator of a diffusion process, whereas, the direct approach \cite{kigami89, kigamipcf} proposed by Kigami, constructs a Laplacian on a class of self-similar fractals that includes the Sierpi\'nski gasket, as a limit of renormalized difference operators. A great deal of literature is built around the study of analysis on different types of fractal sets over the last few decades, \cite{alonsoteplyaev, alonsostrichartzetal, derfeletal, fuku_shima, listrichartz, strichartz, zhou}.
\medskip

\noindent
Generalizing these concepts to an abstract non-fractal setting of the shift space, Denker \textit{et al.} \cite{denker} developed Dirichlet forms on the quotient spaces of the shift space and obtained a corresponding Laplacian. (See \cite{fukushima} for general theory on Dirichlet forms and Laplacian on Hilbert spaces). A shift space is a symbolic description of self-similar fractal sets like Cantor set or Sierpi\'nski gasket. For a finite symbol set $S := \{1,\,2,\,\cdots,\,N \}$ with $N \ge 2$, a one sided full shift space ($\Sigma_{N}^{+}, \sigma)$ is the space of sequences over symbol set $S$ given by,
\[ \Sigma_{N}^{+} := \left\lbrace x = (x_{1}\, x_{2}\, \cdots) : x_{i} \in S, \ \forall i \ge 1\right\rbrace,\]
along with the shift map $\sigma : \Sigma_{N}^{+} \longrightarrow \Sigma_{N}^{+}$, defined as, $\sigma(\,(x_{1}\,x_{2}\,\cdots)\,) = (x_{2}\,x_{3}\,\cdots )$. The discrete topology on $S$ induces a product topology on $\Sigma_{N}^{+}$, under which it is a compact, totally disconnected and perfect metrizable space. Such a distinctive topology makes it interesting to study the analysis on the shift space. Detailed description of the shift space, its dynamical properties and its numerous applications in different fields can be found in \cite{bks, bmn, shannon}. In \cite{arxiv}, we followed Kigami's approach to construct a Laplacian $\Delta$ on the full shift space as a renormalized limit of difference operators $H_{m}$ on certain finite subsets $V_{m}$ of $\Sigma_{N}^{+}$. In sections \eqref{laplacian} and \eqref{energy section}, we summarise these concepts derived in \cite{arxiv}. In section \eqref{energy section}, we further develop some important properties of the finite Dirichlet forms which are induced by the difference operators.
\medskip

\noindent       
The main element involved in the analytical construction of the Laplacian on self-similar sets is the \emph{energy} $\mathcal{E}$ (resistance or Dirichlet form). It is a symmetric bilinear non-negative definite form obtained as the limit of the finite Dirichlet forms. Energy gives rise to an intrinsic effective resistance metric on the underlying fractal set. This metric determines the topology to develop such theory of analysis on a fractal set, independent of its Euclidean embedding. It so happens that, in case of post-critically finite (p.c.f.) self-similar sets, the effective resistance and the Euclidean metric are compatible. Interested readers may refer to \cite{kigamimetric, kigamifrac} for a detailed study of the topic. However, we proved in \cite{arxiv}, that in case of the shift space, the resistance metric does not yield the complete framework of $\Sigma_{N}^{+}$ to develop the analysis. Therefore all the analysis to be carried out further is in the framework of the standard topology on $\Sigma_{N}^{+}$, independent of the effective resistance metric.
\medskip

\noindent
As the difference operators $H_{m}$ induce finite Dirichlet forms $\mathcal{E}_{H_{m}}$ on the finite sets $V_{m}$, it is then natural to explore the relation between the Laplacian $\Delta$ and the energy form $\mathcal{E}$. In case of the p.c.f. self-similar sets, if $u$ and $f$ are continuous functions with $u$ having finite energy, then $\Delta u = f$ if $\mathcal{E}(u,v) = - \int f \, v \, \mathrm{d} \mu$ for all continuous functions $v$ having finite energy and vanishing on the boundary and $\mu$ being an appropriate self-similar probability measure. Such a formulation of the Laplacian is termed as \emph{weak formulation}, due to its clear resemblance with the weak definition of the classical Laplacian. In this paper, we attempt to address this problem in the abstract setting of the shift space. Towards that end, in section \eqref{weak formulation section}, we first conceptualize the idea of a boundary in a totally disconnected space $\Sigma_{N}^{+}$ and propose an analogous weak formulation of the Laplacian taking the various boundary sets into consideration. We restrict the study of the Laplacian to a smaller domain $\mathcal{D}$, than the one considered in \cite{arxiv}. This new domain of the Laplacian is split into subdomains corresponding to various boundary sets. We prove that on each of these subdomains, the weak formulation of the Laplacian agrees with its strong definition. Furthermore, we provide a complete characterization of the harmonic functions in section \eqref{weak formulation section}. 
\medskip

\noindent
The last two sections in this paper focus on solving an analogous Neumann boundary value problem for the Laplacian on $\Sigma_{N}^{+}$. In section \eqref{neumann derivative}, we give a definition of the Neumann derivative of the functions at the boundary points. We establish that the Neumann derivative exists for all the functions in the domain of the Laplacian, $\mathcal{D}$. We further obtain the Gauss-Green's formula relating the Laplacian and the Neumann derivative of a function in $\mathcal{D}$. For a given function $f \in \mathcal{C}(\Sigma_{N}^{+}),$ we provide a sufficient condition for the existence of a solution to the equation
\[ \Delta u \ = \ f, \] 
under Neumann boundary conditions, in section \eqref{boundary problem}.

\section{Laplacian on $\Sigma_{N}^{+}$}
\label{laplacian}
Let us begin by summarizing the basic concepts of the Laplacian in the setting of the symbolic space $\Sigma_{N}^{+}$ developed in \cite{arxiv}. Recall the definition of the one-sided full shift space $(\Sigma_{N}^{+}, \sigma)$ from the previous section. The distance $d$ between any two points $x,y$ in the shift space depends only on the first position where the two sequences disagree, as given below.
\[ d(x, y)\  := \ \frac{1}{2^{\, \rho(x, y)}},\ \ \ \text{where}\ \ \rho(x, y)\ :=\ \min \{ i : x_{i} \ne y_{i} \}\ \ \text{with}\ \ \rho(x, x)\ :=\ \infty. \] 
The cylinder sets of any length $m \ge 1$ in $\Sigma_{N}^{+}$ are defined by 
\[ [p_{1}\, \cdots\, p_{m}]\  := \ \left\lbrace x \in \Sigma_{N}^{+} : x_{1} = p_{1}\, ,\, \cdots\, ,\, x_{m} = p_{m} \right\rbrace, \] 
where the initial $m$ co-ordinates are fixed. The equidistributed Bernoulli measure $\mu$ on $\Sigma_{N}^{+}$ is given by, $ \mu ([p_{1}\, \cdots\, p_{m}]) = \frac{1}{N^{m}}$. The inverse of the shift map $\sigma$ has $N$ branches given by $\sigma_{l} : \Sigma_{N}^{+} \longrightarrow [l]$, for each $l \in S$. These inverse branches give rise to a self-similar structure on the shift space. Consider the set of fixed points of $\sigma$, namely $V_{0} = \left\lbrace \dot{(1)}\, ,\, \dot{(2)}\, ,\, \cdots\, ,\, \dot{(N)}\, \right\rbrace $, where for any $l \in S$, by $(\dot{l})$ we mean the constant sequence $ (l\, l\, \cdots\, ) \in \Sigma_{N}^{+}$. Further, for each $m \ge 1$, the $ m$-th order pre-image of $V_{0}$ is given inductively by $V_{m} := \bigcup\limits_{l\, \in\, S} \sigma_{l}\, (V_{m - 1})$. Any point $p$ in $V_{m}$ is of the form $p = (p_{1}\, \cdots\, p_{m}\, p_{m + 1}\, p_{m+1}\, \cdots\, )$ and for simplicity, it is denoted by $(p_{1}\,  \cdots\, p_{m}\, \dot{p}_{m + 1} )$. In particular, for a point $p \in V_{m} \setminus V_{m - 1}$, we have  $p_{m} \ne p_{m + 1}$. $\{ V_{m} \}_{m \ge 0}$ forms an increasing sequence of subsets of $\Sigma_{N}^{+}$. The set $V_{*} := \bigcup\limits_{m\, \ge\, 0} V_{m}$ is dense in the space $\Sigma_{N}^{+}$, {\it i.e.,} for any $x = (x_{1}\, x_{2}\, \cdots\, ) \in \Sigma_{N}^{+}$, the sequence of points $(\dot{x}_{1}) \in V_{0};\ (x_{1}\, \dot{x}_{2}) \in V_{1};\ \cdots;\ (x_{1}\, x_{2}\, \cdots\, x_{m}\, \dot{x}_{m + 1}\, ) \in V_{m}  $ and so on, converges to $x$. The set $V_{m}$ is a finite set of cardinality $N^{m+1}$. On each $V_{m}$, an equivalence relation $\sim_{m}$ is defined as follows. The points $p,\, q \in V_{m}$ are $m $-related, denoted by $p\sim_{m}q$, if and only if $p_{i} = q_{i}$ for all $1 \le i \le m$. $q$ is also called as $m$-neighbour of $p$ in $V_{m}$. Any two points in $V_{0}$ are $0$-related. The $m$-equivalence class of $p \in V_{m}$ is denoted by $[p_{1}\,p_{2}\, \cdots\, p_{m} ]|_{V_{m}} $. For any point $p \in V_{m}$, its deleted neighbourhood in $V_{m}$ is the set $\mathcal{U}_{p,\, m} = \left\lbrace q \in V_{m} \ | \ q \sim_{m} p, \ q \ne p  \right\rbrace  $ consisting of its $N-1$ points, which are its $m$-neighbours in $V_{m}$.
\medskip

\noindent
For each $m \ge 0$, let $\ell(V_{m}) : = \left\lbrace u \ | \ u : V_{m} \longrightarrow \mathbb{R}  \right\rbrace$, be the set of all real valued functions on $V_{m}$. The standard inner product on $\ell(V_{m})$, denoted by $\left\langle \cdot , \cdot \right\rangle$, is defined as, $\left\langle u,v \right\rangle = \sum\limits_{p\, \in \,V_{m}} u(p)\, v(p)$, for $u,v \in \ell(V_{m})$. The difference operator on $\ell(V_{m})$ is a linear operator $H_{m} : \ell(V_{m}) \longrightarrow  \ell(V_{m})$ defined inductively as follows. First, for $u \in \ell(V_{0})$ and $(\dot{l}) \in V_{0}$,
\[ H_{0} (u) (\dot{l})\ := \ -\, (N - 1)\, u\, (\dot{l}) + \sum\limits_{\substack{k\, \in\, S \\ k\, \ne\, l}} u\, (\dot{k}). \]
Now, let $m \ge 1$, $u \in \ell(V_{m})$ and $p \in V_{m}$. Define $ \kappa_{p} := \min \, \{ j\, : \, p \in V_{j} \}$ to be the index of the smallest set from collection $\{ V_{m} \}_{m \ge 0}$, which contains the point $p$. Note that if $\kappa_{p} \ge 1$ then $p \in V_{\kappa_{p}} \setminus V_{\kappa_{p}-1} $. Then define $H_{m}$ as
\begin{eqnarray*}
H_{m} (u) (p) \ & := \  
\begin{cases}
- (N - 1)\, u(p)\, +\, \sum\limits_{q \, \in \,\mathcal{U}_{p, m}} u(q)\ \qquad \qquad \qquad \quad \ &\text{if} \ \ p \in V_{m} \setminus V_{m - 1}, \\
H_{m - 1} (u|_{V_{m-1}}) (p) + \left[ -\, (N - 1)\, u(p)\, + \sum\limits_{q \, \in \, \mathcal{U}_{p, m}} u(q) \right] \ &\text{if} \ \ p \in V_{m-1}.
\end{cases} \\
& =  H_{\kappa_{p}} u (p) \ + \sum\limits_{j\,= \, \kappa_{p}+1}^{m} \ \sum \limits_{q \, \in \, \mathcal{U}_{p,\,j}} (u(q) - u(p)).  \qquad \qquad \qquad \qquad \qquad \qquad \qquad \qquad 
\end{eqnarray*}
The difference operator $H_{m}$ gives the total difference between the functional values at a point $p$ and its $j$-neighbours in sets $V_{j}$, for all $\kappa_{p} \le j \le m$. Let $\mathcal{C}(\Sigma_{N}^{+})$ denote the set of all real valued continuous functions on $\Sigma_{N}^{+}$. For $u, \, f \in \mathcal{C}(\Sigma_{N}^{+})$, we say $\Delta u = f$ if, 
\begin{equation}
\label{laplaciandef}
\lim_{m \to \infty} \ \max_{p\, \in\, V_{m} \setminus V_{m - 1}} \left| \frac{H_{m} u (p)}{\mu([p_{1} p_{2} \cdots p_{m + 1}])}\, - f(p) \right| = 0.
\end{equation} 
The operator $\Delta$ is known as the \emph{Laplacian}. The set of all continuous functions $u$ for which $\Delta u$ exists, is called as the domain of the Laplacian, and is denoted by $D_{\mu}$. The image of a function $u$ under the Laplacian can be written pointwise as,
\[ f(x) = \lim \limits_{m \to \infty} N^{m+1} H_{m}u (p^{m}), \]  
where $x \in \Sigma_{N}^{+}$ and $\{p^{m}\}_{m \ge 0} $ is a sequence of points such that $ p^{m} \in V_{m}\setminus V_{m-1}$, converging to $x$.

\section{Energy}
\label{energy section}
In this section, we recall the notion of finite Dirichlet forms and energy on the shift space, as introduced in \cite{arxiv}. Since each $H_{m}$ is a symmetric linear operator on $\ell(V_{m})$ with $V_{m}$ being finite, it induces a natural symmetric bilinear form known as \emph{Dirichlet form} $\mathcal{E}_{H_{m}}$ on $\ell(V_{m})$, which is given by,
\begin{equation}
\label{Dirichlet form on V_m}
\mathcal{E}_{H_{m}} (u,v)\ := \ -\left\langle u, H_{m} v  \right\rangle \ = \ -\sum\limits_{p \,\in \, V_{m}} u(p) \,H_{m} v(p) \ \ \text{ for } u, \, v \in \ell( V_{m}).
\end{equation}
See \cite{kigaminetwork} for the concepts of Dirichlet forms on sequence of finite resistance networks. We denote $\mathcal{E}_{H_{m}} (u, u)$ by $\mathcal{E}_{H_{m}}(u)$ for simplicity. The Dirichlet form satisfies the following.
\begin{enumerate}
\item  $\mathcal{E}_{H_{m}}(u) \ge 0$ for all $u \in \ell(V_{m})$, 
\item  $\mathcal{E}_{H_{m}}(u) = 0$ if and only if $u$ is constant on $V_{m}$ and 
\item For any $u \in \ell(V_{m}),\ \mathcal{E}_{H_{m}}(u) \ge \mathcal{E}_{H_{m}}(\bar{u})$ where $\bar{u}$ is defined by 
\begin{equation}
\label{bar u}
\bar{u} (p)\ \ :=\ \ \begin{cases}
1 & \text{if  } \  u(p) \geq 1, \\
u(p) & \text{if  } \   0 < u(p) < 1, \\
0  &  \text{if  } \  u(p) \leq 0.
\end{cases}
\end{equation}
\end{enumerate} 
Following are some alternate expressions for $\mathcal{E}_{H_{m}}$ which will prove to be useful in most of the calculations in the subsequent sections. Let $u, \, v \in \ell(V_{m})$.
\begin{eqnarray*}
\mathcal{E}_{H_{m}}(u, v)\ & = & \ \frac{1}{2} \sum\limits_{p,q \,\in \, V_{m}} (H_{m})_{p q}  \left( u(p) -u(q) \right) \, \left( v(p) -v(q) \right) \\
& = & \ \frac{1}{2} \ \sum\limits_{i = 0}^{m} \ \sum\limits_{p \, \in \,V_{i}} \ \sum\limits_{q \, \in \, \mathcal{U}_{p,\,i}}  \left( u(p) -u(q) \right) \, \left( v(p) -v(q) \right). 
\end{eqnarray*} 
This sequence of the Dirichlet forms $\{ \mathcal{E}_{H_{m}} \}_{m \ge 0}$, defines a symmetric bilinear operator $\mathcal{E}$ on $\mathcal{C}(\Sigma_{N}^{+})$ as,
\[\mathcal{E}(u, v) \ = \ \lim\limits_{m \to \infty} \mathcal{E}_{H_{m}}\,(u|_{V_{m}}, v|_{V_{m}}). \]
For $v = u$, if the limit $\mathcal{E}(u) = \lim\limits_{m \to \infty} \mathcal{E}_{H_{m}}\,(u|_{V_{m}}) $ is finite, then $\mathcal{E}(u)$ is known as the \textit{energy} of $u$. The set of all such continuous functions having finite energy is called as the \emph{domain of energy}, denoted by  $dom \, \mathcal{E}$. Any function $u_{m} \in \ell(V_{m})$ can be uniquely extended to a continuous function $h : \Sigma_{N}^{+} \longrightarrow \mathbb{R}$ taking constant values on the $(m+1)$-long cylinder sets. Such an extension $h$ of $u_{m}$, is called as an \emph{energy minimizer extension}, due to the fact that this extension is the only function that minimizes the energy in the sense,
\[ \mathcal{E}(h) \ = \ \mathcal{E}_{H_{m}} (u_{m}) \ = \ \min \left\lbrace \mathcal{E}_{H_{n}} (v )\ |\ n > m, \  v \in \ell(V_{n}),\ v|_{V_{m}} = u_{m} \right\rbrace. \]
This property makes the sequence of difference operators $\{H_{m}\}_{m \ge 0}$ compatible, in the sense of Kigami. For a point $p \in V_{m}$, consider the characteristic function $\chi_{q} \in \ell(V_{m})$ of a point $q \in V_{m}$,
\begin{eqnarray*}
\chi_{p} (q) =  
\begin{cases}
1 & \text{ if } \ q = p,\\
0 & \text{ otherwise. }
\end{cases}     
\end{eqnarray*}
The energy minimizer extension of $\chi_{p} $ be denoted by $\chi_{p}^{m} : \Sigma_{N}^{+} \longrightarrow \mathbb{R}$ and is given by,
\begin{equation} 
\label{uhe}
\chi_{p}^{m}\ \ =\ \  
\begin{cases}
1 &  \text{ on } \  [p_1\, p_2\, \cdots\, p_{m+1}] \\
0 &  \text{ elsewhere. }
\end{cases} 
\end{equation}
In general, any function taking constant values on cylinder sets of some particular length is called as \emph{energy minimizer}. These functions are the basic simple functions on $\Sigma_{N}^{+}$ and they play an important role in the discussion. Since an energy minimizer $h$ takes constant values on cylinder sets of length, say $m+1$, for some $m \ge 0$, it satisfies $\Delta h = 0$ and is a harmonic function. Moreover, $ D_{\mu}  \, \subset \,  dom \, \mathcal{E} \, \subset \, \mathcal{C}(\Sigma_{N}^{+})$ with the inclusions being dense. The density can be seen through the fact that harmonic functions, in particular the energy minimizers belong to $D_{\mu}$ and also form a dense subset of $\mathcal{C}(\Sigma_{N}^{+})$. For instance, let $u \in \mathcal{C}(\Sigma_{N}^{+})$ and for each $m \ge 0$, define energy minimizers $u_{m} $  taking constant values on cylinder sets of length $m+1$ as,
\begin{equation}
\label{seq of harm fn}
u_{m} := \sum\limits_{p\, \in \, V_{m}} u(p)\, \chi_{p}^{m}.
\end{equation}
Then $u_{m}$ converges to $u$ uniformly as $m \to \infty$. All the proofs can be found in \cite{arxiv}. 
\medskip

\noindent
We now investigate some more interesting properties of energy $\mathcal{E}$ and $ dom\, \mathcal{E}$. Let us first look into the energy of an energy minimizer $h$, taking constant values on cylinder sets of length $m+1$, for some $m \ge 0$. By definition, 
\[\mathcal{E}(h,u) = \frac{1}{2} \, \lim\limits_{n \to \infty} \sum\limits_{i=0}^{n}\ \sum\limits_{p \, \in \, V_{i}} \ \sum\limits_{q \, \in \, \mathcal{U}_{p,\,i}} \Big( h(p)-h(q) \Big) \Big( u(p)-u(q) \Big).\]
Observe that for each $i \ge m+1$, $h(p) - h(q) = 0$, for all $p \in V_{i}$ and $q \, \in \, \mathcal{U}_{p,\,i}$, since each such pair of points $p$ and $q$ belongs to the same cylinder set of length $m+1$. Therefore, $ \mathcal{E}(h,u) = \mathcal{E}_{H_{m}}(h,u)$. Further, for $u = h$, it follows that $\mathcal{E}(h) = \mathcal{E}_{H_{m}} (h)$. By an abuse of notation, by $ \mathcal{E}_{H_{m}} (h)$ we mean $\mathcal{E}_{H_{m}} (h|_{V_{m}})$ hereafter. We have thus proved the following.
\begin{proposition} \label{energy_min_energy} 
If $h$ is an energy minimizer, then there exists $m \ge 0$ such that $\mathcal{E}(h,u) = \mathcal{E}_{H_{m}}(h,u)$. In particular $\mathcal{E}(h) = \mathcal{E}_{H_{m}} (h)$.
\end{proposition}

\noindent
The energy $\mathcal{E}(u) = \lim\limits_{m \to \infty} \sum\limits_{p, q\, \in \, V_{m}} (H_{m})_{pq} \big( u(p)-u(q) \big)^{2} $ is always non-negative, since all the summands are non-negative. Also, $\mathcal{E}(u) = 0$ if and only if $u$ is constant. We say that a sequence of functions $\{v_{n}\}_{n \ge 0}$ converges to a function $v$ in energy if, $v_{n}$ converges uniformly to $v$ and $\mathcal{E}(v_{n}-v) \to 0$ as $n \to \infty$. 

\begin{lemma}
\label{convergence in energy}
If $u \in dom \, \mathcal{E}$, then the sequence of functions $\{ u_{m} \}_{m \ge 0} $ as defined in equation \eqref{seq of harm fn}, converges to $u$ in energy. 
\end{lemma}
\begin{proof}
The uniform convergence of $u_{m}$ to $u$ is already established. Now, for $u \in dom \, \mathcal{E}$, consider the functions as defined in equation \eqref{seq of harm fn}. Since each $u_{m}$ is an energy minimizer, by the above proposition, we have $u_{m} \in dom \, \mathcal{E}$ and, 
\begin{equation}
\label{energy_1}
 \mathcal{E}(u_{m}) = \mathcal{E}_{H_{m}}(u_{m}).
\end{equation}
\noindent
Since $u|_{V_{m}} = u_{m}|_{V_{m}}$, we know that the energy of $u$ is,
\[ \mathcal{E}(u) = \lim\limits_{m \to \infty} \mathcal{E}_{H_{m}} (u_{m}). \]
For each $m \ge 0$ observe that, 
\begin{align}
\mathcal{E} (u,\, u_{m}) & = \frac{1}{2} \ \lim\limits_{n \to \infty} \ \sum\limits_{i = 0}^{n} \ \sum\limits_{p \, \in \, V_{i}} \ \sum\limits_{q \, \in \, \mathcal{U}_{p,\,i}} \left( u(p) -u(q) \right)  \left( u_{m}(p) -u_{m}(q) \right) \nonumber\\
& = \frac{1}{2} \ \sum\limits_{i = 0}^{m} \ \sum\limits_{p \, \in \, V_{i} } \ \sum\limits_{q \, \in \, \mathcal{U}_{p,\,i}} \left( u_{m}(p) -u_{m}(q) \right)^{2}\nonumber \\
& = \mathcal{E}_{H_{m}}(u_{m}) \label{energy_2}.
\end{align}
Consider,
\begin{align*}
\mathcal{E}(u - u_{m}) & = \frac{1}{2} \ \lim\limits_{n \to \infty} \ \sum\limits_{i = 0}^{n} \ \sum\limits_{p \, \in \, V_{i} } \ \sum\limits_{q \, \in \,  \mathcal{U}_{p,\,i}}  \big( \,(u - u_{m})(p) - (u - u_{m})(q) \, \big)^{2} \\
& =  \frac{1}{2} \ \lim\limits_{n \to \infty} \ \sum\limits_{i = 0}^{n} \ \sum\limits_{p \, \in \, V_{i}} \ \sum\limits_{q \, \in \, \mathcal{U}_{p,\,i}}  \big( \,(u(p) - u(q)) - (u_{m}(p) - u_{m}(q)) \, \big)^{2} \\
& = \frac{1}{2} \ \lim\limits_{n \to \infty} \ \sum\limits_{i = 0}^{n} \ \sum\limits_{p \, \in \, V_{i} } \ \sum\limits_{q \, \in \, \mathcal{U}_{p,\,i}} \left[ (u(p) - u(q))^{2} \, + \, (u_{m}(p) - u_{m}(q))^{2} \right] \\
& \hspace{2cm} +  \lim\limits_{n \to \infty} \ \sum\limits_{i = 0}^{n} \ \sum\limits_{p \, \in \, V_{i} } \ \sum\limits_{q \, \in \, \mathcal{U}_{p,\,i}} (u(p) - u(q))(u_{m}(p) - u_{m}(q)) \\
& = \mathcal{E}(u) + \mathcal{E}(u_{m}) - 2 \,\mathcal{E}(u,u_{m}).
\end{align*}
Using equations \eqref{energy_1} and \eqref{energy_2}, we obtain,
\[ \mathcal{E}(u - u_{m})  = \mathcal{E}(u) - \mathcal{E}_{H_{m}} (u_{m}).\]
Therefore,
\[ \lim\limits_{m \to \infty} \mathcal{E}(u - u_{m}) \ = \ \mathcal{E}(u) - \lim\limits_{m \to \infty} \mathcal{E}_{H_{m}} (u_{m}) \ = \ 0. \]
\end{proof}

\section{Weak formulation of the Laplacian}
\label{weak formulation section}
As the shift space is totally disconnected and has topological dimension $0$, the topological notion of a boundary is of little significance. This means that we have complete control over the choice of the boundary sets. Our natural candidate for a boundary is any of the sets $V_{M}$ for $M \ge 0$, on which the finite difference operators are defined. This choice is motivated from the study of analysis on finite sets, according to which, a set $V_{m}$ is considered as the boundary of the set $V_{m+1}$. Having set $V_{M}$ as a boundary, define the space of finite energy functions vanishing on the boundary $V_{M}$ as,
\begin{equation*}
dom_{M}\,\mathcal{E}\ := \ \left\lbrace u \in dom \,\mathcal{E} \ : \ u|_{V_{M}} = 0  \right\rbrace.
\end{equation*}
We now propose the following weak formulation of the Laplacian.
\begin{definition}
Let $u \in dom \, \mathcal{E}$ and $f \in \mathcal{C}(\Sigma_{N}^{+})$. We say $\Delta\, u = f$ in weak sense, if there exist some $M \ge 0$ such that 
\begin{equation}
\label{weak laplacian}
\mathcal{E} (u,v) \ = \ - \int\limits_{\Sigma_{N}^{+}} f\, v \,\mathrm{d}\mu, \quad \text{for all } \ v \in dom_{M}\,\mathcal{E}.
\end{equation}
\end{definition}
For each $M \ge 0$, consider the following subsets of the domain of the Laplacian $D_{\mu}$. 
\begin{eqnarray}
\label{new domain of laplacian}
\mathcal{D}_{M} & := & \Bigg\{ u \in D_{\mu} : \exists \, f \in \mathcal{C}(\Sigma_{N}^{+}) \ \text{with} \ \Delta u = f \ \text{satisfying} \nonumber \\ 
& & \hspace{+1cm} \lim_{m \to \infty} \max_{p\, \in\, V_{m} \setminus V_{M}} \left| \frac{H_{m} u (p)}{\mu([p_{1} p_{2} \cdots p_{m + 1}])}\, - f(p) \right| = 0 \Bigg\}. 
\end{eqnarray} 
Each $\mathcal{D}_{M}$ is a linear subspace of $D_{\mu}$ and $\Delta|_{\mathcal{D}_{M}} : \mathcal{D}_{M} \longrightarrow \mathcal{C}(\Sigma_{N}^{+})$ is a linear operator. In each $\mathcal{D}_{M}$, the maximum is taken over the entire set $V_{m} \setminus V_{M}$ and not just $V_{m} \setminus V_{m-1}$ which was the case in $D_{\mu}$. So these are the functions on $\Sigma_{N}^{+}$, with $\Delta u = f$, with $V_{M}$ being the boundary, for which the renormalized sequence $N^{m+1} H_{m} u$ converges to $f$ stronger than for the functions in $D_{\mu}$.  
\medskip

\noindent
It is simple to see that $\{ dom_{m}\, \mathcal{E} \}_{m \ge 0}$ and $\{ \mathcal{D}_{m}\}_{m \ge 0}$ form a nested sequence of sets as,
\[ dom_{0}\, \mathcal{E} \ \supset \ dom_{1}\, \mathcal{E} \ \supset \ dom_{2}\, \mathcal{E} \ \supset \ \cdots \ , \ \text{and}, \ \ \mathcal{D}_{0} \ \subset \ \mathcal{D}_{1} \  \subset \ \mathcal{D}_{2} \ \subset \ \cdots.  \]
Consider $\mathcal{D} := \bigcup\limits_{M \ge 0} \mathcal{D}_{M}$. For the purpose of studying the weak formulation, we focus on the Laplacian operator restricted to the domain $\mathcal{D}$, $\Delta|_{\mathcal{D}} :\mathcal{D} \longrightarrow \mathcal{C}(\Sigma_{N}^{+})$.
\medskip

\noindent
Let $u \in \mathcal{D}_{M}$ and $\Delta u = f$. Our aim is to show that $\Delta u = f $ in the weak sense too. That is, $\ \mathcal{E}(u,v) = - \int\limits_{\Sigma_{N}^{+}} f\, v \,\mathrm{d}\mu$, for all $v \in dom_{M}\,\mathcal{E}$. Each $dom_{M}\, \mathcal{E}$ acts as a set of test functions for the Laplacian $\Delta|_{\mathcal{D}_{M}}$. First we make the following two observations.
\begin{lemma}
\label{chi_p^m as test function}
For $m \ge M+1$ and $p \in V_{m} \setminus V_{M}$, $\chi_{p}^{m} \in dom_{M}\, \mathcal{E}$.
\end{lemma}
\begin{proof}
If $p = (p_{1}\,\cdots \, p_{m}\, \dot{p}_{m+1} ) \in V_{m} \setminus V_{M},$ then $ \ M+1 \le \kappa_{p} \le m$ and $p_{\kappa_{p}} \neq p_{\kappa_{p}+1}$. Now, if $x \in V_{M}$ then $x_{i} = x_{i+1} $, for all $i \ge M+1$. In particular, $x_{\kappa_{p}} = x_{\kappa_{p}+1}$. Clearly, $x \notin [p_{1}\,p_{2} \cdots p_{m+1}]$ and $\chi_{p}^{m} (x) = 0$, proving $\chi_{p}^{m} \in dom_{M}\, \mathcal{E}$. 
\end{proof}

\begin{lemma}
\label{E(u,chi= -H_m u (p))}
For $m \ge M+1$ and $p \in V_{m} \setminus V_{M}$, $\ \mathcal{E}(u, \chi_{p}^{m}) = - H_{m} u (p)$.
\end{lemma}
\begin{proof}
Since $\chi_{p}^{m}$ is an energy minimizer taking constant values on cylinder sets of length $m+1$, by proposition \eqref{energy_min_energy}, \[\mathcal{E}(u, \chi_{p}^{m}) = \mathcal{E}_{H_{m}} (u, \chi_{p}^{m}) = \frac{1}{2} \sum\limits_{i = 0}^{m} \ \sum\limits_{x \, \in \, V_{i}} \ \sum\limits_{q \, \in \, \mathcal{U}_{x,\,i}} \left( u(x) -u(q) \right)  \left( \chi_{p}^{m}(x) -\chi_{p}^{m}(q) \right).\]
Note that for any $x \in V_{m}$, $\chi_{p}^{m}(x) = 1 $ if and only if $x = p$.
We know that $p \in V_{\kappa_{p}} \setminus V_{\kappa_{p}-1}$ for some $\kappa_{p}$ such that $M+1 \le \kappa_{p} \le m$. Then for all $\kappa_{p} \le i \le m$ and $q \in \mathcal{U}_{p,\,i}$, we have $ \chi_{p}^{m}(q) = 0$. Due to the fact that the roles of $p$ and $q$ are reversible, we obtain,
\[ \mathcal{E}(u, \chi_{p}^{m}) \ = \ \sum\limits_{i = \kappa_{p}}^{m} \ \sum\limits_{q \, \in \, \mathcal{U}_{p,\,i}} \left( u(p) -u(q) \right) \ = \ - H_{m} u (p).  \]\end{proof}
\medskip

\noindent
For every $m \ge M+1$, consider the function $v_{m} := \sum\limits_{p\, \in \, V_{m}} v(p)\, \chi_{p}^{m}$ as defined in equation \eqref{seq of harm fn} which converge to $v$ uniformly, as $m \to \infty$. Since $v|_{V_{M}} = 0$, $\ v_{m} = \sum\limits_{p\, \in \, V_{m} \setminus V_{M}} v(p)\, \chi_{p}^{m}$. Lemma \eqref{chi_p^m as test function} implies that, $v_{m} \in dom_{M}\,\mathcal{E}$. Note that each $v_{m}$ is an energy minimizer and takes constant values on cylinder sets of length $m+1$. Then by proposition \eqref{energy_min_energy} we have,
\begin{eqnarray}
\lim\limits_{m \to \infty} \mathcal{E}(u,v_{m}) & = & \lim\limits_{m \to \infty} \mathcal{E}_{H_{m}}(u,v_{m}) \notag \\
& = & \lim\limits_{m \to \infty} \ \frac{1}{2} \ \sum\limits_{i = 0}^{m} \ \sum\limits_{p \, \in \, V_{i}} \ \sum\limits_{q \, \in \, \mathcal{U}_{p,\,i}} \left( u(p) -u(q) \right)  \left( v(p) -v(q) \right) \nonumber\\
& = & \mathcal{E}(u,v). \nonumber 
\end{eqnarray} 
Also, by the dominated convergence theorem, 
\[ \lim\limits_{m \to \infty} \int\limits_{\Sigma_{N}^{+}} f\, v_{m} \,\mathrm{d}\mu \ = \ \int\limits_{\Sigma_{N}^{+}} f\, v \,\mathrm{d}\mu. \]
Therefore it is enough to prove that 
\[ \displaystyle{ \lim\limits_{m \to \infty} \Big| \mathcal{E}(u,v_{m}) \, + \, \int\limits_{\Sigma_{N}^{+}} f\, v_{m} \,\mathrm{d}\mu \Big| \ = \ 0}. \]
Now, recalling the basic definition of integration, we may write
\[ \int\limits_{\Sigma_{N}^{+}} f\, v_{m} \,\mathrm{d}\mu \ = \ \lim\limits_{n \to \infty} \sum\limits_{ \left\lbrace [p_{1}\cdots p_{n+1}]\, : \, p_{i}\, \in\, S \right\rbrace }  f(t) \, v_{m}(t) \, \mu([p_{1}\cdots p_{n+1}]), \]
where $t \in [p_{1}\cdots p_{n+1}]$ is a tag for the partition of $\Sigma_{N}^{+}$ by all cylinder sets of length $n+1$. Choose $t = (p_{1}\cdots p_{n}\, \dot{p}_{n+1}) \in V_{n}$. Each such point $t \in [p_{1}\cdots p_{n+1}]$ and the cylinder sets of length $n+1$ are in one to one correspondence. Moreover if $t \in V_{M}$ then $v_{m}(t) = 0$ . Therefore the above sum can be rewritten as,
\begin{equation}
\label{expression for integral}
\int\limits_{\Sigma_{N}^{+}} f\, v_{m} \,\mathrm{d}\mu \ = \ \lim\limits_{n \to \infty} \  \frac{1}{N^{n+1}}\ \sum\limits_{p \,\in \, V_{n} \setminus V_{M}} f(p)\, v_{m}(p).
\end{equation} 
Using the definition of Dirichlet forms on $V_{m}$ as given in equation \eqref{Dirichlet form on V_m}, we get
\begin{eqnarray}
\label{expression for energy}
\mathcal{E}(u,v_{m}) & = & \lim\limits_{n \to \infty} \mathcal{E}_{H_{n}} (u, v_{m}) \notag \\
& = & - \lim\limits_{n \to \infty} \sum\limits_{p \, \in \, V_{n}\setminus V_{M}} (H_{n}u)(p) \, v_{m}(p).
\end{eqnarray}
Making use of \eqref{expression for integral} and \eqref{expression for energy} we obtain,
\begin{eqnarray*}
\lim\limits_{m \to \infty} \Big| \mathcal{E}(u,v_{m}) & + &  \int\limits_{\Sigma_{N}^{+}} f\, v_{m} \,\mathrm{d}\mu \Big|  \\
& = & \lim\limits_{m \to \infty} \ \lim\limits_{n \to \infty} \ \Big| \frac{1}{N^{n+1}} \sum\limits_{p \, \in \, V_{n} \setminus V_{M}} f(p)\, v_{m}(p) - \sum\limits_{p \, \in \, V_{n} \setminus V_{M}} (H_{n}u)(p) \, v_{m}(p)   \Big| \\
& \le & \lim\limits_{m \to \infty} \ \lim\limits_{n \to \infty} \ \frac{1}{N^{n+1}} \sum\limits_{p \, \in \, V_{n} \setminus V_{M}} \ \sup\limits_{x \, \in \, \Sigma_{N}^{+}} |v(x)| \ \max\limits_{p \, \in \, V_{n} \setminus V_{M}} \left| f(p) - N^{n+1} (H_{n} u)(p)  \right| \\
& = & \sup\limits_{x \, \in \, \Sigma_{N}^{+}} |v(x)| \ \lim\limits_{m \to \infty} \ \left\lbrace \lim\limits_{n \to \infty} \ \max\limits_{p \, \in \, V_{n} \setminus V_{M}} \left| f(p) - N^{n+1} (H_{n} u)(p) \right| \frac{N^{n+1} - N^{M+1}}{N^{n+1}} \right\rbrace \\
& = & 0,
\end{eqnarray*}
as, $\frac{N^{n+1} - N^{M+1}}{N^{n+1}} \le 1$ and thus $\Delta u = f$ holds in the weak sense. This suggests that we call the formulation of the Laplacian for the functions in the new domain $\mathcal{D}_{M}$ as described in equation \eqref{new domain of laplacian}, as the \textit{strong formulation}. Moreover, we have the following theorem.
\begin{theorem}
The strong and the weak formulation of the Laplacian agree on $\mathcal{D}_{M}$. 
\end{theorem}
\begin{proof}
We have already established the part strong $\implies$ weak, in the discussion before stating the theorem. To prove weak $\implies$ strong on any $\mathcal{D}_{M}$, let $u \in \mathcal{D} \subset dom \, \mathcal{E}$ and $\Delta u = g$ in the weak sense. There exists $M \ge 0 $ such that the equation \eqref{weak laplacian} holds.
Since $\chi_{p}^{m} \in dom_{M}\, \mathcal{E}$,  $\ \mathcal{E} (u,\chi_{p}^{m}) = - \int\limits_{\Sigma_{N}^{+}} g\, \chi_{p}^{m} \,\mathrm{d}\mu$. Then, by lemma \eqref{E(u,chi= -H_m u (p))}, $H_{m} u (p) = \int\limits_{[p_{1} \cdots p_{m+1}]} g\, \mathrm{d}\mu$. Now consider,
\begin{eqnarray*}
\max\limits_{p \in V_{m}\setminus V_{M}} \left| N^{m+1} H_{m} u (p) - g(p)   \right| & = & \max\limits_{p \in V_{m}\setminus V_{M}} \left| \frac{1}{\mu([p_{1} \cdots p_{m + 1}])} \, \int\limits_{[p_{1} \cdots p_{m+1}]} g\, \mathrm{d}\mu \, - \, g(p)   \right| \\
& \to & 0, \ \ \text{as} \ \ m \to \infty.
\end{eqnarray*}
The convergence follows directly from the Lebesgue differentiation theorem, see \cite{heinonen}. Therefore, $u \in \mathcal{D}_{M}$ and $\Delta|_{\mathcal{D}_{M}} u = g$ in the strong sense.
\end{proof}
\medskip

\noindent
We have already established in \cite{arxiv}, that every energy minimizer is a harmonic function. The weak formulation of the Laplacian further guarantees that these are the only harmonic functions in $\mathcal{D}$.
\begin{theorem}
Every harmonic function in $\mathcal{D}$ is an energy minimizer.
\end{theorem} 
\begin{proof}
Let $h \in \mathcal{D}$ be a harmonic function, that is, $\Delta h = 0$. By the weak formulation of the Laplacian, there exists $M \ge 0$ such that
$\mathcal{E} (h,v) =  0$ for all $v \in dom_{M}\,\mathcal{E}$. Using lemma \eqref{chi_p^m as test function}, for each $m > M$ and any $x = (x_{1}\,\cdots \, x_{m}\, \dot{x}_{m+1}) \in V_{m} \setminus V_{m-1}$, $\ \chi_{x}^{m} \in \ dom_{M} \,\mathcal{E}.$
Therefore, $\mathcal{E} (h,\chi_{x}^{m}) =  0$. By proposition \eqref{energy_min_energy},
\begin{align*}
\mathcal{E} (h,\chi_{x}^{m}) \ & = \ \mathcal{E}_{H_{m}} (h,\chi_{x}^{m}) \\
& = \ \frac{1}{2} \sum\limits_{i = 0}^{m} \sum\limits_{p \, \in \, V_{i} } \ \sum\limits_{q \, \in \, \mathcal{U}_{p,\,i}} \left( h(p) -h(q) \right)  \left( \chi_{x}^{m}(p) - \chi_{x}^{m}(q) \right) \\
& = \ \sum\limits_{q \,\in \, \mathcal{U}_{x,\,m}} \left( h(x) -h(q) \right)\\
& = \ 0.
\end{align*}
Recall that $\mathcal{U}_{x,\,m} $ contains $N-1$ points which are the $m-$neighbours of $x$ in $V_{m}$. Let us denote these neighbours by $q^{1},\,q^{2},\cdots, q^{N-1} \, \in \,\mathcal{U}_{x,\,m} $. Amongst these $N-1$ points, let $ q^{1},\cdots, q^{N-2} \, \in \,V_{m}\setminus V_{m-1}$ and $q^{N-1} = (x_{1}\,\cdots \, x_{m-1}\, \dot{x}_{m})\in V_{m-1}$. Following the same method for each of the characteristic functions $\chi_{q^{1}}^{m},\, \chi_{q^{2}}^{m},\cdots, \chi_{q^{N-2}}^{m} $, we obtain the following system of $N-1$ equations.
\begin{align*}
-(N-1)\, h(x)\,+\,h(q^{1})\,+\, \cdots \,+\,h(q^{N-2})\,+\,h(q^{N-1}) \ & = \ 0\\
h(x)\,-(N-1)\, h(q^{1})\,+\, \cdots \,+\,h(q^{N-2})\,+\,h(q^{N-1}) \ & = \ 0 \\
& \vdots \\
h(x)\,+\,h(q^{1})\,+\, \cdots \,-(N-1)\, h(q^{N-2})\,+\,h(q^{N-1}) \ & = \ 0. 
\end{align*} 
Solving the above simultaneous system, we get
\[ h(x) \,=\,h(q^{1}) \,=\, \cdots \, = \,h(q^{N-1}),  \]
which implies that $h$ assumes constant values on cylinder sets of length $m$ for all $m > M$. This essentially means that $h$ is an energy minimizer taking constant values on cylinder sets of length $M+1$, proving the theorem.
\end{proof}
\medskip

\noindent
A function $u \in \mathcal{D}$, belongs to $\mathcal{D}_{M}$ for some $M \ge 0$. In that case $V_{M}$ is chosen to be a boundary for $\Sigma_{N}^{+}$. Thus, for a function $u$ in $\mathcal{D}$, there is a naturally associated boundary set $V_{M}$ for some $M \ge 0$, which is particular to the function $u$. Note that we do not require a fixed boundary for all the functions in $\mathcal{D}$. Further, since the set $V_{0} $ is contained in $V_{M}$ for all $M \ge 0$, it is enough to have the boundary values specified only on the set $V_{0}$, for the purpose of solving any kind of boundary value problem. In \cite{arxiv}, the Dirichlet boundary value problem was solved for the operator $\Delta : D_{\mu} \longrightarrow \mathcal{C}(\Sigma_{N}^{+})$. Interestingly, the solution to this problem obtained therein can be seen to be belonging to $\mathcal{D}$. This point will be apparent following the discussion in section \eqref{boundary problem}. Therefore, we will restrict our study of the Laplacian only on the domain $\mathcal{D}$. The improved version of the Dirichlet boundary value problem is stated as follows. We do not present the proof here, as the arguments run along the same lines as in \cite{arxiv}.

\noindent 
\begin{theorem}
For any given $f \in \mathcal{C}(\Sigma_{N}^{+}) \text{ and } \zeta \in \ell(V_{0})$,  there exists a continuous function $u \in \mathcal{D}_{0}$ such that the following holds:
\[\Delta u = f \ \ \ \text{subject to} \ \ \ u|_{V_{0}} = \zeta. \]
This solution is unique upto harmonic functions taking value $0$ on $V_{0}$.
\end{theorem}

\section{Neumann derivative}
\label{neumann derivative}
In this section, we obtain a relation between the Laplacian and the energy of a function in a more general form than the weak formulation defined in the previous section. For that purpose, we define the concept of the Neumann derivative of a function in $\mathcal{D}$ at the boundary points. As a direct corollary to this result, we derive a relation between the Laplacian and its normal derivatives at the boundary points. This formula has a natural resemblance with the Gauss-Green's formula in classical calculus.
%
%

\begin{lemma}
Let $u \in \mathcal{D}$ and $V_{M}$ be the corresponding boundary for $u$. Then for each $p \in V_{M}$, the limit $\lim\limits_{m \to \infty} \, H_{m} u (p) $ exists.
\end{lemma}
\begin{proof}
Let $u \in \mathcal{D}$ with $V_{M}$ being the corresponding boundary. Then there exists a positive constant $C \in \mathbb{R}$, and $M_{0}\ge M$ such that for all $i \ge M_{0}$, $p \in V_{i}$ and $q \in \mathcal{U}_{p,\, i} $, $\left| u(q)-u(p)   \right| \le \frac{C}{N^{i+1}}$. Let $p \in V_{M}$ and $m > n \ge M_{0}$.
\begin{eqnarray*}
\left| H_{m} u (p) - H_{n} u (p)  \right| & = & \left| \sum\limits_{i= n+1}^{m} \  \sum\limits_{q \, \in \, \mathcal{U}_{p,\, i}} \big( u(q)-u(p) \big) \right| \\
& \le & C \, \frac{(N-1)}{N^{n+1}} \sum\limits_{i= 1}^{m-n} \frac{1}{N^{i}}
\end{eqnarray*}
which converges to $0$ as $m , n \to \infty$. Thus for each $p \in V_{M}$, the sequence $\{ H_{m} u (p)\}_{m \ge 0}$ is Cauchy and hence convergent.
\end{proof}

\begin{definition}
Let $u \in \mathcal{D}$ and $V_{M}$ be the corresponding boundary for $u$. The Neumann derivative of $u$ at $p \in V_{M}$, denoted by $(du) (p)$,  is defined as, 
\begin{equation}
\label{Neumann derivative}
(du) (p) \ := \ \lim\limits_{m \to \infty} -\, H_{m} u (p).
\end{equation}   
\end{definition}

\begin{theorem}
Let $u \in \mathcal{D}$, $v \in dom \, \mathcal{E}$ and $V_{M}$ be the corresponding boundary for $u$. Then,
\begin{equation}
\label{Gauss-Green1}
\mathcal{E}(u,v) \ = \ \sum\limits_{p\, \in\, V_{M}} v(p)\,(du) (p) \,-\, \int\limits_{\Sigma_{N}^{+}} (\Delta u) \,v \, \mathrm{d}\mu.
\end{equation}
Moreover, 
\begin{equation}
\label{Gauss-Green2}
\sum\limits_{p\, \in\, V_{M}} (du) (p) \ = \ \int\limits_{\Sigma_{N}^{+}} \Delta u  \, \mathrm{d}\mu.
\end{equation}
\end{theorem}

\begin{proof}
Let $u \in \mathcal{D}$ and $v \in dom \, \mathcal{E}$. By the definition of the Dirichlet form, for all $m \ge M$,
\begin{align*}
\mathcal{E}_{H_{m}}(u,v) \ &= \ -\left\langle v,\, H_{m}u  \right\rangle \\
&= \ - \sum\limits_{p \, \in \, V_{M}} v(p)\, H_{m}u(p) \ - \sum\limits_{p \, \in \, V_{m} \setminus V_{M}} v(p)\, H_{m}u(p).
\end{align*}
For each $v \in dom\,\mathcal{E}$, consider $\tilde{v} \in dom_{M}\, \mathcal{E}$ such that $\tilde{v}|_{\,\Sigma_{N}^{+} \setminus V_{M}} = v$ and $\tilde{v}|_{V_{M}} = 0$. Then,
\begin{equation*}
\mathcal{E}_{H_{m}}(u,\tilde{v}) \ = \ -\sum\limits_{p \, \in \, V_{m} \setminus V_{M}} \tilde{v}(p)\, H_{m}u(p).
\end{equation*}
Therefore,
\begin{equation*}
\mathcal{E}_{H_{m}}(u,v) \ = \ - \sum\limits_{p \, \in \, V_{M}} v(p)\, H_{m}u(p) \, + \, \mathcal{E}_{H_{m}}(u,\tilde{v}).
\end{equation*}
Taking limit as $m\to \infty$, using the weak formulation for $\Delta u$ and the definition of the normal derivative of $u$, we obtain,
\begin{align*}
\mathcal{E}(u,v) \ &= \ - \sum\limits_{p \, \in \, V_{M}} v(p)\, \left( \lim\limits_{m \to \infty} H_{m}u(p) \right) \, + \, \mathcal{E}(u,\tilde{v}) \\
&= \ \sum\limits_{p\, \in\, V_{M}} v(p)\,(du) (p) \,-\, \int\limits_{\Sigma_{N}^{+}} (\Delta u) \,\tilde{v} \, \mathrm{d}\mu.
\end{align*}
Since $\tilde{v}$ agrees with $v$ except for a set of $\mu$ measure zero, $\int\limits_{\Sigma_{N}^{+}} (\Delta u) \,\tilde{v} \, \mathrm{d}\mu \ = \ \int\limits_{\Sigma_{N}^{+}} (\Delta u) \,v \, \mathrm{d}\mu$. Thus, \eqref{Gauss-Green1} holds. Further, \eqref{Gauss-Green2} can be obtained for $v \equiv 1$ in \eqref{Gauss-Green1}. 
\end{proof}
\medskip

\noindent
The analogous Gauss-Green's formula on $\Sigma_{N}^{+}$ is easily obtained as a corollary to the above theorem due to the symmetry of the energy $\mathcal{E}$.
\begin{corollary}[Gauss-Green's formula]
For $u, v \in \mathcal{D}_{M}$,
\begin{equation*}
\int\limits_{\Sigma_{N}^{+}} \left( v\, \Delta u \, - \, u \, \Delta v   \right)  \, \mathrm{d}\mu \ = \ \sum\limits_{p\, \in\, V_{M}} \big( v(p) (du) (p) \, - \, u(p) (dv) (p) \big) .
\end{equation*}
\end{corollary}
%

\section{Neumann Boundary Value Problem}
\label{boundary problem}
We begin this section by recalling some of the concepts of Green's function and Green's operator on $\Sigma_{N}^{+}$, developed in \cite{arxiv}. These tools will help us in proving the existence of the solution to the equation $\Delta u = f$ under Neumann boundary conditions. Consider the operator  $G_{m}\, : \,  \ell (V_{m} \setminus V_{m - 1})  \longrightarrow  \ell (V_{m} \setminus V_{m - 1})$ defined by, 
\begin{eqnarray}
\label{G_m} 
(G_{m})_{pq}& = & 
\begin{cases}
\frac{2}{N} &\text{ if } q=p,\\
\frac{1}{N} &\text{ if } q \sim_{m} p,\\
0            &\text{ otherwise, } 
\end{cases}  
\end{eqnarray}
where $(G_{m})_{pq}$ denotes the value $(G_{m} \chi_{q}) (p)$. The \emph{Green's function} on $\Sigma_{N}^{+}$, $ \ g : \Sigma_{N}^{+} \times \Sigma_{N}^{+} \longrightarrow \mathbb{R} \cup \{ \infty \}$ is defined as,
\begin{eqnarray*}
g(x,y)\ \  = \ \ \begin{cases}
\sum\limits_{m\,=\,1}^{\rho(x,y)-1} \ \sum\limits_{r,s \,\in \,V_{m}  \setminus V_{m-1}} (G_{m})_{rs}\, \chi_{r}^{m}(x)\, \chi_{s}^{m}(y) &  \text{ if } \ \rho(x,y) > 1, \\
0 &  \text{ if } \ \rho(x,y) = 1,
\end{cases}
\end{eqnarray*}
where $ \rho(x,y)$ is the first instance where $x$ and $y$ disagree. Let $\mathcal{L}^{1}(\Sigma_{N}^{+})$ be the space of $\mu$-integrable functions on $\Sigma_{N}^{+}$. The \emph{Green's operator} on $\mathcal{L}^{1}(\Sigma_{N}^{+}) $ is an integral operator whose kernel is the Green's function. It is defined as,
\begin{equation*}
G_{\mu}f(x)\ \ := \ \ \int_{\Sigma_N^+ \setminus \{x \} } g(x,y)\, f(y) \,\mathrm{d} \mu(y) \quad \text{ for } \ f \in \mathcal{L}^{1}(\Sigma_{N}^{+}) . 
\end{equation*}  
\begin{lemma} \cite{arxiv}
\label{prop_g.o.}
For any $f \in \mathcal{C}(\Sigma_{N}^{+})$, the following holds.
\begin{enumerate}
\item $(G_{\mu} f)|_{V_{0}} \ = \ 0.$
\item \label{H_mG_mu}
For any $n \ge 1$ and $p \in V_{n}\setminus V_{n-1}$,
\begin{equation*} 
H_{n} (G_{\mu} f)(p) \ = \ - \int\limits_{\Sigma_{N}^{+}} \chi_{p}^{n}\, f \,\mathrm{d} \mu. 
\end{equation*}
\item \label{prop3} $G_{\mu}f \in D_{\mu}$ and $\Delta\, (G_{\mu}f) = -f $.
\end{enumerate}
\end{lemma}

\begin{theorem}
\label{N.D.of G.O.}
For any $f \in \mathcal{C}(\Sigma_{N}^{+})$ and $M \ge 0$,
\begin{equation*}
d \, (G_{\mu}f) (p) \ = \ \begin{cases}
0 \ & \text{if} \ p \in V_{M}\setminus V_{0}\\
-\int\limits_{[p_{1}]} f \, \mathrm{d}\mu & \text{where} \ p = (\dot{p}_{1}) \in V_{0}.
\end{cases}
\end{equation*} 
\end{theorem}
\begin{proof}
Let $p \in V_{M}\setminus V_{0}$. Then $0 < \kappa_{p} \le M$ and $p = (p_{1}\,\cdots\,p_{\kappa_{p}}\,\dot{p}_{\kappa_{p}+1})$ with $p_{\kappa_{p}} \neq p_{\kappa_{p}+1}$. By statement \eqref{H_mG_mu} of lemma \eqref{prop_g.o.} we have,
\begin{equation}
\label{H_i_p}
H_{\kappa_{p}} (G_{\mu} f) (p)\ = \ - \int\limits_{\Sigma_{N}^{+}} \chi_{p}^{\kappa_{p}}\, f \,\mathrm{d} \mu. 
\end{equation}
Along the same lines, we claim that, for each $m > \kappa_{p}$,
\begin{equation}
\label{extensionH_mG_mu}
H_{m} (G_{\mu} f )(p) \ =  \ - \int\limits_{\Sigma_{N}^{+}} \chi_{p}^{m}\, f \,\mathrm{d} \mu .
\end{equation}
Let $m = \kappa_{p}+1$. $H_{\kappa_{p}+1}$ can be written in terms of $H_{\kappa_{p}}$ as,
\begin{equation}
\label{step1induction}
H_{\kappa_{p}+1} (G_{\mu} f )(p) \ = \ H_{\kappa_{p}} (G_{\mu} f) (p) \ + \, \left[ -\, (N - 1)\, (G_{\mu} f) (p)\ + \sum\limits_{q\, \in\, \mathcal{U}_{p,\, \kappa_{p}+1}} (G_{\mu} f)(q) \right].
\end{equation}
Let us denote the points in $\mathcal{U}_{p,\, \kappa_{p}+1}$ by, $\mathcal{U}_{p,\, \kappa_{p}+1} \, = \, \{ r^{1},\,r^{2},\,\cdots,\, r^{N-1} \} \subset V_{\kappa_{p}+1} \setminus V_{\kappa_{p}}$. Then,
\begin{eqnarray}
\label{eq1}
-\, (N - 1)\, (G_{\mu} f) (p)\ & + & \sum\limits_{r\, \in\, \mathcal{U}_{p,\, \kappa_{p}+1}} (G_{\mu} f)(r) \nonumber \\
 =  \ \int\limits_{\Sigma_{N}^{+} \setminus \{ p,\,r^{1}, \,\cdots,\, r^{N-1}  \}} & &\left[ - (N-1)\,g(p,y)\,+\, g(r^{1},y)\,+\,\cdots +\, g(r^{N-1},y)\right]\,f(y)\,\mathrm{d}\mu(y).
\end{eqnarray}
Consider,
\begin{align*}
-\,(N-1)\,g(p,y)\,&+\, g(r^{1},y)\,+\,\cdots \,+\, g(r^{N-1},y)  \\
= -\,&(N-1) \left[ \sum\limits_{m\,=\,1}^{\kappa_{p}} \ \sum\limits_{s,t\,\in \,V_{m} \setminus V_{m-1}} (G_{m})_{st}\, \chi_{s}^{m}(p)\, \chi_{t}^{m}(y)  \right]\   \\ 
&+\sum\limits_{m\,=\,1}^{\kappa_{p}+1}\ \sum\limits_{s,t \,\in\, V_{m} \setminus V_{m-1}} (G_{m})_{st}\, \chi_{s}^{m}(r^{1})\, \chi_{t}^{m}(y)  \\
&+ \cdots \\
&+\ \sum\limits_{m\,=\,1}^{\kappa_{p}+1} \ \sum\limits_{s,t \,\in\, V_{m} \setminus V_{m-1}} (G_{m})_{st}\, \chi_{s}^{m}(r^{N-1})\, \chi_{t}^{m}(y). 
\end{align*}
We know that $p,\,r^{1},\,r^{2},\,\cdots,\, r^{N-1} $ agree on the first $\kappa_{p}+1$ positions. So for all $m \le \kappa_{p}$, all the terms in the above sum are equal and get canceled. Only the terms corresponding to $m = \kappa_{p}+1$ remain. That is,
\begin{align*}
-\,(N-1)\,g(p,y)\,&+\, g(r^{1},y)\,+\,\cdots \,+\, g(r^{N-1},y)  \\
= & \sum\limits_{ s,t \,\in\, V_{\kappa_{p}+1} \setminus V_{\kappa_{p}} } (G_{\kappa_{p}+1})_{st}\, \chi_{s}^{\kappa_{p}+1}(r^{1})\, \chi_{t}^{\kappa_{p}+1}(y)   \\
& + \cdots \\
& + \sum\limits_{s,t \,\in \,V_{\kappa_{p}+1} \setminus V_{\kappa_{p}}} (G_{\kappa_{p}+1})_{st}\, \chi_{s}^{\kappa_{p}+1}(r^{N-1})\, \chi_{t}^{\kappa_{p}+1}(y) \\
= & \sum\limits_{ t\,\in\, V_{\kappa_{p}+1} \setminus V_{\kappa_{p}} } (G_{\kappa_{p}+1})_{r^{1}t}\, \chi_{t}^{\kappa_{p}+1}(y) \ + \ \cdots + \ \sum\limits_{t \,\in \,V_{\kappa_{p}+1} \setminus V_{\kappa_{p}}} (G_{\kappa_{p}+1})_{r^{N-1}t}\, \chi_{t}^{\kappa_{p}+1}(y). \\
\end{align*}
The second step here follows due to the fact that for each $1 \le i \le N-1$, $r^{i} \in V_{\kappa_{p}+1} \setminus V_{\kappa_{p}}$ and $\chi_{s}^{\kappa_{p}+1}(r^{i}) = 1$ if and only if $s = r^{i}$. Substituting for the values $(G_{\kappa_{p}+1})_{q^{i}s}$ as given in equation \eqref{G_m}, we get,
\[ -\,(N-1)\,g(p,y)\,+\, g(r^{1},y)\,+\,\cdots \,+\, g(r^{N-1},y) \ = \ \chi_{r^{1}}^{\kappa_{p}+1}(y) \,+\,\cdots\,+\, \chi_{r^{N-1}}^{\kappa_{p}+1}(y).\]
Define the set $P_{\kappa_{p}+1} := [ p_{1}\,\cdots\,p_{\kappa_{p}}\,p_{\kappa_{p}+1}] \setminus \left( \, [p_{1}\,\cdots\,p_{\kappa_{p}}\,p_{\kappa_{p}+1}\,p_{\kappa_{p}+1}] \cup \mathcal{U}_{p,\, \kappa_{p}+1} \,\right) $. Using equation \eqref{H_i_p} and above calculations in equation \eqref{step1induction} we obtain,
\begin{eqnarray*}
H_{\kappa_{p}+1} (G_{\mu} f )(p) \ & = & \ - \int\limits_{\Sigma_{N}^{+}} \chi_{p}^{\kappa_{p}}\, f \,\mathrm{d} \mu \ + \ \int\limits_{\Sigma_{N}^{+} \setminus \{ p,\,r^{1}, \,\cdots,\, r^{N-1} \} } \left( \chi_{r^{1}}^{\kappa_{p}+1}(y) \,+\,\cdots\,+\, \chi_{r^{N-1}}^{\kappa_{p}+1}(y) \right) \, f(y) \, \mathrm{d} \mu \\
& = & - \int\limits_{[ p_{1}\,\cdots\,p_{\kappa_{p}}\,p_{\kappa_{p}+1}]} f \,\mathrm{d} \mu \ + \  \int\limits_{P_{\kappa_{p}}}  f \,\mathrm{d} \mu \\
& = & - \int\limits_{[ p_{1}\,\cdots\,p_{\kappa_{p}}\,p_{\kappa_{p}+1}\,p_{\kappa_{p}+1}]} f \,\mathrm{d} \mu \\
& = &  - \int\limits_{\Sigma_{N}^{+}} \chi_{p}^{\kappa_{p}+1}\, f \,\mathrm{d} \mu.
\end{eqnarray*}
As $ \{r^{1},\,r^{2},\,\cdots,\, r^{N-1} \}$ is only a finite set, it has measure $0$. Similarly by the method of induction we can prove that for each $m > \kappa_{p}$,
\[ H_{m} (G_{\mu} f )(p) \ =  \ - \int\limits_{\Sigma_{N}^{+}} \chi_{p}^{m}\, f \,\mathrm{d} \mu \  = \ \int\limits_{[ p_{1}\,\cdots\,p_{m}\,p_{m+1} ]} f \,\mathrm{d} \mu.  \]
As $m \to \infty$, the cylinder sets $[p_{1}\,\cdots \, p_{m} \, p_{m+1}] \to \{p\}$. Thus, $H_{m} (G_{\mu} f )(p) \to 0$ and thus $d(G_{\mu} f) (p) = 0$.
\medskip

\noindent
Now, let $p = (p_{1}\, p_{2}\,\cdots ) = (\dot{p}_{1})\in V_{0}$, that is for all $i \ge 1, \ p_{i} = p_{1}$. Clearly $H_{0}(G_{\mu} f) = 0$, as $(G_{\mu} f)|_{V_{0}} = 0$. By the inductive definition of $H_{m}$ we have,
\[ H_{m} (G_{\mu} f) (p) \ = \ \sum\limits_{n = 1}^{m} \sum\limits_{q \, \in \, \mathcal{U}_{p,\,n}} (G_{\mu} f) (q). \]
Following similar arguments employed in the first part of the proof we get,
\begin{eqnarray*}
 H_{1} (G_{\mu} f) (p) \ & = & \int\limits_{\Sigma_{N}^{+} \setminus \{ r^{1}, \,\cdots,\, r^{N-1} \} } \left( \chi_{r^{1}}^{1}(y) \,+\,\cdots\,+\, \chi_{r^{N-1}}^{1}(y) \right) \, f(y) \, \mathrm{d} \mu \\
 & = & \int\limits_{[ p_{1} ] \setminus \mathcal{U}_{p,\,1}} f \,\mathrm{d} \mu -  \int\limits_{[ p_{1}\,p_{2} ] } f \,\mathrm{d} \mu,
\end{eqnarray*}
where $  \mathcal{U}_{p,\,1} \, = \, \{ r^{1},\,r^{2},\,\cdots,\, r^{N-1} \} \subset V_{1}\setminus V_{0}$. Again by induction it follows that,
\begin{equation*}
H_{m} (G_{\mu} f) (p) \ = \int\limits_{[ p_{1} ] \setminus \bigcup\limits_{i=1}^{m}\mathcal{U}_{p,\,i}} f \,\mathrm{d} \mu -  \int\limits_{[ p_{1}\,\cdots \,p_{m+1} ] } f \,\mathrm{d} \mu, 
\end{equation*}
For each $m \ge 1$, the set $\bigcup\limits_{i=1}^{m}\mathcal{U}_{p,\,i}$ is countable and hence of measure $0$. Also as $m \to \infty$,  $[ p_{1}\,\cdots \,p_{m+1} ] \to \{p\}$. Therefore, $d(G_{\mu} f) (p) = - \int\limits_{[p_{1}]} f \, \mathrm{d}\mu  $.

\end{proof}

\medskip

\noindent
We conclude the section by stating the Neumann boundary value problem for the Laplacian $\Delta$ and provide a sufficient condition for the existence of its solution. 
\begin{theorem}
Let $f \in \mathcal{C}(\Sigma_{N}^{+})$ and $\xi \in \ell(V_{0}) $. If $\xi(p) = \int\limits_{[p_{1}]} f \, \mathrm{d}\mu $, then there exists $u \in \mathcal{D}_{0}$ satisfying,
\[ \Delta u \,= \, f \ \ \ \text{subject to,} \ \ \ (du)(p) \, = \, \xi (p)\ \text{for} \ p \in  V_{0}. \]
%
\end{theorem} 
\begin{proof}
Consider $u = -G_{\mu} f$. Then by property \eqref{prop3} of lemma \eqref{prop_g.o.}, $u$ satisfies the differential equation $\Delta u = f$. This $u$ also satisfies the boundary conditions by virtue of theorem \eqref{N.D.of G.O.}. Further, in order to prove $G_{\mu} f \in \mathcal{D}_{0}$, let $m > 0$ and $p \in V_{m} \setminus V_{0}$. Using equation \eqref{extensionH_mG_mu} we get,
\begin{eqnarray*}
\left| N^{m+1}\, H_{m} G_{\mu}f(p) \,+\, f(p)  \right| & = & \left| -\int\limits_{\Sigma_{N}^{+}}N^{m+1} \, \chi_{p}^{m}(y) \,(f(p)\,-\,f(y))\,\mathrm{d} \mu(y)    \right| \\
& \le & \int\limits_{[p_{1}\,p_{2}\,\cdots\, p_{m+1}]}N^{m+1}\, \left| f(p)-f(y) \right| \,\mathrm{d} \mu(y)\\
& \le & \epsilon_{m} \ \ \to \ \  0, \ \ \text{as} \ \ m \to \infty,
\end{eqnarray*}
where $\epsilon_{m} := \sup\limits_{y \, \in \, [p_{1}\,p_{2}\,\cdots\, p_{m+1}]} \left| f(p)-f(y) \right| $. The convergence $\epsilon_{m} \to 0$ follows from uniform continuity of $f$.
\end{proof}

\bibliographystyle{plainnat}

\bigskip 
\bigskip
\bigskip

\noindent 
\textsc{Shrihari Sridharan} \\ 
Indian Institute of Science Education and Research Thiruvananthapuram (IISER-TVM), \\ 
Maruthamala P.O., Vithura, Thiruvananthapuram, INDIA. PIN 695 551. \\
{\tt shrihari@iisertvm.ac.in}  
\bigskip \\ 

\noindent 
\textsc{Sharvari Neetin Tikekar} \\ 
Indian Institute of Science Education and Research Thiruvananthapuram (IISER-TVM), \\ 
Maruthamala P.O., Vithura, Thiruvananthapuram, INDIA. PIN 695 551. \\
{\tt sharvai.tikekar14@iisertvm.ac.in}

\end{document}